\documentclass[a4paper,12pt]{article}
\usepackage{amsmath}
\usepackage{amssymb}
\usepackage{amsthm}
\usepackage[english]{babel}
\usepackage{hyperref}

\newtheorem{theor}{Theorem}[section]
\newtheorem{defi}[theor]{Definition}
\newtheorem{lemma}[theor]{Lemma}

\newcommand{\Sign}{\mathrm{Sign}}
\newcommand{\Irr}{\mathrm{Irr}}

\begin{document}

\begin{center}
{\Large 
Sign conjugacy classes of the alternating groups}

Lucia Morotti
\end{center}

\begin{abstract}

A conjugacy class $C$ of a finite group $G$ is a sign conjugacy class if every irreducible character of $G$ takes value 0, 1 or -1 on $C$. In this paper we classify the sign conjugacy classes of alternating groups.

\end{abstract}

\section{Introduction}

We will begin this paper by giving the definition of sign conjugacy class for an arbitrary finite group.

\begin{defi}
Let $G$ be a finite group. A conjugacy class of $G$ is a sign conjugacy class of $G$ if every irreducible character of $G$ takes values 0, 1 or -1 on $C$.
\end{defi}

In \cite{b7} Olsson considered sign conjugacy classes of $S_n$ in order to answer a question from Isaacs and Navarro for $S_n$, using the following property of sign conjugacy classes:

\begin{theor}
Let $C$ be a sign conjugacy class of a finite group $G$ and define $\chi^+$ and $\chi^-$ by
\[\chi^+:=\sum_{{\chi\in\Irr(G):}\atop{\chi(C)>0}}\chi\mbox{\hspace{24pt}and\hspace{24pt}}\chi^-:=\sum_{{\chi\in\Irr(G):}\atop{\chi(C)<0}}\chi.\]
Then $\chi^+$ and $\chi^-$ are characters of $G$ differing only on $C$.
\end{theor}

In \cite{b7} Olsson also formulated a conjecture about sign conjugacy classes of $S_n$ which was proven in \cite{m2}. In this paper we will classify sign conjugacy classes of $A_n$, proving that they are closely related to those of $S_n$, as would be to expect due to the relationship of irreducible characters of $S_n$ and $A_n$.

As we will be working with alternating (and symmetric) groups, we will refer to partitions instead of conjugacy classes. This leads to the following two definitions:

\begin{defi}
A partition $\gamma$ is a sign partition if it is the cycle partition of a sign conjugacy class of $S_n$.

If $\gamma$ is an even partition then $\gamma$ is an $A_n$-sign partition if it is the cycle partition of a sign conjugacy class of $A_n$.
\end{defi}

For example $(3)$ is a sign partition which is not an $A_n$-sign partition, while $(2,2)$ is an $A_n$-sign partition which is not a sign partition. In general however $A_n$-sign partitions are almost always sign partitions and almost all even sign partitions are $A_n$ sign partitions, as will be seen from Theorem \ref{t3} and Theorem 1.3 of \cite{m2}.

For $\gamma$ consisting of odd distinct parts it is clear that $C_{\gamma,+}$ is a sign conjugacy class of $A_n$ if and only if $C_{\gamma,-}$ is also such a conjugacy class (later it will be proved that no such conjugacy class is a sign conjugacy class). In particular any conjugacy class of $A_n$ is a sign conjugacy class exactly when its cycle partition is an $A_n$-sign partition.

In order to describe $A_n$-sign partitions we will need the two following sets of partitions:

\begin{defi}
We define $\Sign$ to be the subsets of partitions consisting of all partitions $(\gamma_1,\ldots,\gamma_r)$ for which there exists an $s$, $0\leq s\leq r$ such that the following holds:
\begin{itemize}
\item
$\gamma_i>\gamma_{i+1}+\ldots+\gamma_r$ for $1\leq i\leq s$,

\item
$(\gamma_{s+1},\ldots,\gamma_r)$ is one of the following partitions:
\begin{itemize}
\item
$()$, $(1,1)$, $(3,2,1,1)$ or $(5,3,2,1)$,

\item
$(a,a-1,1)$ with $a\geq 2$,

\item
$(a,a-1,2,1)$ with $a\geq 4$,

\item
$(a,a-1,3,1)$ with $a\geq 5$.
\end{itemize}
\end{itemize}
\end{defi}

\begin{defi}
We define $\overline{\Sign}$ by:
\begin{eqnarray*}
\overline{\Sign}&:=&\{(1,1,1),(2,2),(2,2,1),(5,4,3,2,1)\}\\
&&\hspace{6pt}\cup\{(a,a-1,4,1):a\geq 6\}\\
&&\hspace{6pt}\cup\{(a,a-3,2,1,1):a=6\mbox{ or }a\geq 8\}\\
&&\hspace{6pt}\cup\{(a,a-5,3,2,1):9\leq 4\leq 10\mbox{ or }a\geq 12\}\\
&&\hspace{6pt}\cup\{(a,b,a-b+1):b+1\leq a\leq 2b-2\}\\
&&\hspace{6pt}\cup\{(a,b,a-b-1,1):b+2\leq a\leq 2b\}.
\end{eqnarray*}
\end{defi}

In \cite{m2} (Theorem 1.3) it was proven that a partition $\gamma$ is a sign partition if and only if $\gamma\in\Sign$, proving a conjecture of Olsson from \cite{b7}. In this paper we will prove the following characterization of $A_n$-sign conjugacy classes.

\begin{theor}\label{t3}
Let $\gamma$ be a partition of $n\geq 2$. Then $\gamma$ is an $A_n$-sign partition if and only if
\[\gamma\in(\Sign\cup\overline{\Sign})\cap\{\mbox{even partitions not consisting of odd distinct parts}\}.\]
\end{theor}

In order to prove Theorem \ref{t3} we will use the following lemmas.

\begin{lemma}\label{l1}
Let $\gamma$ be an $A_n$-sign partition. If $\gamma\not\in\{(1,1,1),(2,2),(2,2,1)\}$ then $\gamma$ does not have repeated parts, except possibly for the part 1, which may have multiplicity 2.
\end{lemma}

\begin{lemma}\label{t2}
Let $\alpha=(\alpha_1,\ldots,\alpha_h)$ be a partition with $h\geq 3$. Assume that $\alpha_1>\alpha_2$, that $\alpha\not\in\Sign$ and that $(\alpha_2,\ldots,\alpha_h)\in\Sign$. Then if $\alpha\not=(5,4,3,2,1)$ we can find a partition $\beta$ of $|\alpha|$ such that $\chi^\beta_\alpha\not\in\{0,\pm 1\}$ and $h_{2,1}^\beta=\alpha_1$.

If $\alpha\not\in\overline{\Sign}$ we can choose $\beta$ to be not self conjugate.
\end{lemma}

These two lemmas will allow us to prove in Section \ref{s5} that, if $\gamma=(\gamma_1,\ldots,\gamma_r)$ is an $A_n$-sign partition, then $(\gamma_i,\ldots,\gamma_r)\in\Sign$ for $i\geq 2$ and $\gamma\in\Sign\cup\overline{\Sign}$, by otherwise constructing partitions $\delta$ which are not self conjugate and which satisfy $\chi^\delta_\gamma\not\in\{0,\pm 1\}$, contradicting then the assumption that $\gamma$ is an $A_n$-partition. In Section \ref{s6} we will prove that partitions consisting of odd distinct parts are not $A_n$-sign partitions (by simply looking at the characters of $A_n$ indexed by the corresponding self conjugate partition), proving then one direction of Theorem \ref{t3}. The other direction will be proved in Section \ref{s7}. Through all of the following $\gamma$ will be a partition of $n\geq 2$.

From Theorem 1.3 of \cite{m2} and from Theorem \ref{t3} we also easily obtain that if $\chi^\delta_\gamma\not\in\{0,\pm 1\}$ and $\gamma\in\overline{\Sign}$ does not consist of distinct odd parts, then $\delta$ is self conjugate. This can be proved also if $\gamma\in\overline{\Sign}$ consist of distinct odd parts, using arguments similar to those from Section \ref{s7}.

Proofs of results about irreducible characters of $S_n$ and $A_n$ and about weights, cores and quotients of partitions can be found in \cite{b1} and \cite{b4}.

\section{Proof of Lemmas \ref{l1} and \ref{t2}}

We will now prove Lemmas \ref{l1} and \ref{t2}. To do this we will use results from \cite{m2} and \cite{b7}. Since most of the partitions considered there are of the form $(a,b,1^c)$ we will first classify in the next lemma which such partitions are self conjugate. In the following if $\beta$ is a partition we will write $\beta'$ for its conjugate.

\begin{lemma}\label{l3}
A partition $(a,b,1^c)$ of $n\geq 2$ is self conjugate if and only if $c=a-2$ and $b\in\{1,2\}$. In particular if $(a,b,1^c)$ is self conjugate then $b+c\in\{a-1,a\}$.
\end{lemma}

\begin{proof}
Let $\beta:=(a,b,1^c)$. If $b=0$ clearly $\beta$ is not self conjugate, as $a=n\geq 2$. So assume that $b\geq 1$. From $\beta_2'\leq 2$ it follows that if $\beta$ is self conjugate then $b\in\{1,2\}$. As clearly $(a,1^{c+1})$ and $(a,2,1^c)$ are self conjugate if and only if $c=a-2$ the lemma follows.
\end{proof}

We will now prove Lemma \ref{l1}.

\begin{proof}[Proof of Lemma \ref{l1}]
Assume that $\gamma$ has a repeated part (with multiplicity at least 3 if this part is 1).

If $\gamma\not\in\{(1,1,1),(2,2)\}$ and $\gamma$ is not of the form $(a,a,1)$ for $a\geq 4$ or of the form $(\delta_1,\ldots,\delta_t,\mu_1,\ldots,\mu_v)$ with $\delta_t>\mu_1$ and
\[\mu\in\{(2^m,1),(2^m,1^2),(3^m,2,1),(3^m,2,1^2),(3^m,1),(3^m,1^2):m\geq 2\},\]
we can find from the proof of Lemma 6 and Theorem 7 of \cite{b7} a partition $\beta$ with $\chi^\beta_\gamma\not=\{0,\pm1\}$ which, from Lemma \ref{l3}, is not self conjugate. It then easily follows that $\gamma$ is not an $A_n$-sign partition. In the not covered cases let
\[\beta=\left\{\begin{array}{ll}
(a+1,a),&\gamma=(a,a,1),\\
(n-2,1^2),&\mu=(2^m,1)\mbox{ or }(2^m,1^2),\\
(n-3,3),&\mu=(3^m,2,1),(3^m,2,1^2)\mbox{ or }(3^m,1),\\
(n-3,1^3),&\mu=(3^m,1^2),
\end{array}\right.\]
where $n=|\gamma|$. It can be easily checked using Lemma \ref{l3} that the above partitions are not self-adjoint unless $\gamma=(2,2,1)$. Also $\chi^\beta_\gamma\not\in\{0,\pm1\}$, for example
\[\chi^{(n-2,1^2)}_{(\delta_1,\ldots,\delta_t,2^m,1)}=\chi^{(2m-1,1^2)}_{(2^m,1)}=-(m-1)\chi^{(3)}_{(2,1)}+\chi^{(1^3)}_{(2,1)}=-m.\]
The lemma then follows.
\end{proof}

We will now prove Lemma \ref{t2}. Most of the work will be in proving that if $\alpha\not\in\overline{\Sign}$ then the partitions $\beta$ constructed in the proof of Theorem 1.6 of \cite{m2} are not self-adjoint. This is always the case apart for $\alpha=(\alpha_2+2a-1,\alpha_2,a,a-1,1)$ with $\alpha_2>2a$ and $a\geq 4$, which will be treated separately.

\begin{proof}[Proof of Lemma \ref{t2}]
We will divide the proof of the lemma in the following cases: 1) $\alpha=(\alpha_2+2a-1,\alpha_2,a,a-1,1)$ with $\alpha_2>2a$ and $a\geq 4$, 2) all other cases. Case 2) will be divided in subcases corresponding to the different cases of the proof of Theorem 1.6 of \cite{m2}.

1) For $\alpha=(\alpha_2+2a-1,\alpha_2,a,a-1,1)$ with $\alpha_2>2a$ and $a\geq 4$ let $\beta:=(\alpha_2+2a,3,1^{\alpha_2+2a-4})$. Since $a\geq 4$ we have that $\beta$ is a partition. Also $h_{2,1}^\beta=\alpha_2+2a-1=\alpha_1$ and $\beta$ is not self-adjoint from Lemma \ref{l3}. As
\[h_{1,3}^\beta=\alpha_2+2a+2-3=\alpha_1,\]
as $a\geq 4$ so that
\[\alpha_2+2a-4>\alpha_2\]
and as any partition of $\alpha_2+2a$ has at most one $\alpha_2$-hooks since $\alpha_2>2a$, we have
\begin{eqnarray*}
\chi^\beta_\alpha&=&(-1)^{\alpha_2+2a-4}\chi^{(\alpha_2+2a)}_{(\alpha_2,a,a-1,1)}-\chi^{(2,2,1^{\alpha_2+2a-4})}_{(\alpha_2,a,a-1,1)}\\
&=&(-1)^{\alpha_2+2a-4}-(-1)^{\alpha_2-1}\chi^{(2,2,1^{2a-4})}_{(a,a-1,1)}\\
&=&(-1)^{\alpha_2+2a-4}-(-1)^{\alpha_2-1+a-1}\chi^{(2,2,1^{a-4})}_{(a-1,1)}\\
&=&(-1)^{\alpha_2+2a-4}-(-1)^{\alpha_2-1+a-1+a-3}\chi^{(1)}_{(1)}\\
&=&(-1)^{\alpha_2}2.
\end{eqnarray*}
So the lemma holds in this case.

2) In each of the following cases $\beta$ is as constructed in \cite{m2} for the corresponding case. Notation used here is as in \cite{m2}.

\begin{itemize}
\item
For
\begin{eqnarray*}
(\alpha_2,\ldots,\alpha_h)&\in&\{(1,1),(3,2,1,1),(5,3,2,1)\}\\
&&\hspace{6pt}\cup\{(a,a-1,1):2\leq a\leq 4\}\\
&&\hspace{6pt}\cup\{(a,a-1,2,1):4\leq a\leq 8\}\\
&&\hspace{6pt}\cup\{(a,a-1,3,1):5\leq a\leq 10\}
\end{eqnarray*}
the lemma can be checked by looking at each single case separately (they are finitely many since $\alpha_1\leq \alpha_2+\ldots+\alpha_h$).

\item
For
\begin{eqnarray*}
(\alpha_2,\ldots,\alpha_h)&\in&\{(a,a-1,1):a\geq 5\}\\
&&\hspace{6pt}\cup\{(a,a-1,2,1):a\geq 9\}\\
&&\hspace{6pt}\cup\{(a,a-1,3,1):a\geq 11\}
\end{eqnarray*}
we can apply Lemma \ref{l3} to results from Section 2 of \cite{m2} and, as $\alpha_1>\alpha_2=a\geq 5$, obtain that if $\beta$ is self conjugate then $\alpha$ is either $(2a,a,a-1,1)$ or $(a+1,a,a-1,1)$, as in all other cases $\beta=(|\alpha|-\alpha_1,b,1^{\alpha_1-b})$ with $|\alpha|-\alpha_1\geq \alpha_1+2$ or $b\geq 4$.

If $\alpha=(a+1,a,a-1,1)$ then $\beta=(a-1,a-1,a-1,4)$ is self conjugate if and only if $a=5$, that is $\alpha=(6,5,4,1)\in\overline{\Sign}$.

If $\alpha=(2a,a,a-1,1)=(2a,a,2a-a-1,1)$. As $a\geq 5$ we have that $\alpha\in\overline{\Sign}$.

\item
For $\alpha$ as in Theorems 3.1 or 3.2 of \cite{m2} we have from Lemma \ref{l3} that if $\beta$ is self conjugate then $x=1$ and $|\alpha|-\alpha_1=\alpha_1$. From $x=1$ we have that $k=h$ and $\alpha_h=1$. From $|\alpha|-\alpha_1=\alpha_1$ we have that $\alpha_1=\alpha_2+\ldots+\alpha_h$, that is $\alpha_1-\alpha_2=\alpha_3+\ldots+\alpha_h$, so that $k=4$. So $\alpha=(\alpha_1,\alpha_2,\alpha_1-\alpha_2-1,1)$. From $\alpha_1-\alpha_2>\alpha_h=1$ and $\alpha_2>\alpha_1-\alpha_2-1$ it follows that $\alpha\in\overline{\Sign}$ when $\beta$ is self-adjoint.

\item
For $\alpha$ as in Theorem 3.3 of \cite{m2}, $\beta$ is not self-adjoint from Lemma \ref{l3}, as
\[\beta_2=\alpha_1-c=\alpha_2+\alpha_{k-1}+\ldots+\alpha_h-c>\alpha_2>2.\]

\item
For $\alpha$ as in Theorem 3.4 of \cite{m2} we have that if $\beta$ is self conjugate then $k=4$ and $(\alpha_{k-1},\ldots,\alpha_h)=(a,a-1,1)$, as in all other cases $\alpha_1\not=|\alpha|-\alpha_1-1$. Here $\alpha_1=\alpha_2+2a-1$ and $\alpha_2>2a$. The case $a\geq 4$ has already been considered in 1). For $a=2$ or $a=3$ instead we have $\alpha\in\overline{\Sign}$.

\item
For $\alpha$ as in Theorems 3.5, 3.9 and 3.11 of \cite{m2}, $\beta$ is not self-adjoint from Lemma \ref{l3}, as $\alpha_1>2$.

\item
For $\alpha$ as in Theorems 3.6, 3.7 and 3.10 of \cite{m2} the lemma follows again from Lemma \ref{l3}, as $\alpha_2>1$.

\item
For $\alpha$ as in Theorem 3.8 of \cite{m2}, if $\beta$ is self-adjoint then $\alpha_1=|\alpha|-\alpha_1+1$. By assumption $|\alpha|-\alpha_1\geq\alpha_2+\alpha_h>\alpha_1$, so that if $\beta$ is self-adjoint then $h=3$ and $\alpha=(\alpha_1,\alpha_2,\alpha_1-\alpha_2+1)$. As $\alpha_1>\alpha_2>\alpha_1-\alpha_2+1$, so that $\alpha_2+1\leq \alpha_1\leq 2\alpha_2-2$, it follows that $\alpha\in\overline{\Sign}$ when $\beta$ is not self conjugate.

\item
For $\alpha$ as in Theorem 3.12 of \cite{m2} we have that $\beta_2=\alpha_3\geq 4=\beta_2'$. In particular $\beta$ is not self conjugate for $\alpha_3\geq 5$. For $\alpha_3=4$ then $\alpha=(\alpha_1,\alpha_1-1,4,1)\in\overline{\Sign}$, as $\alpha_1-1>4$, so that $\alpha_1\geq 6$.

\item
For $\alpha$ as in Theorem 3.13 of \cite{m2} we have that $\alpha_3>\alpha_4\geq \alpha_{h-1}=2$ as $h\geq 5$ and then $\alpha_1>\alpha_2>\alpha_3+\alpha_4\geq 5$. In particular $\alpha_1\geq 7$ and then $\beta_2'=4<\alpha_1-2=\beta_2$ so that $\beta$ is not self conjugate.

\item
For $\alpha$ as in Theorem 3.14 of \cite{m2} if $\beta$ is self-adjoint then $\alpha_{h-1}=3$ and
\[|\alpha|-\alpha_1-\alpha_{h-1}+1=\alpha_1-\alpha_{h-1}-1+\alpha_{h-1}-3+3.\]
So $\alpha_2+\ldots+\alpha_h=\alpha_1+1$, which from the assumptions is equivalent to $\alpha_3+\ldots+\alpha_{h-1}=1$. Always from the assumption this would give
\[3<\alpha_3+\ldots+\alpha_{h-1}=1\]
leading to a contradiction. So $\beta$ is not self-adjoint.
\end{itemize}
\end{proof}

\section{$A_n$-sign partitions are elements of $\Sign\cup\overline{\Sign}$}\label{s5}

Let $\gamma$ be an $A_n$-sign partition. If $r\leq 2$ then clearly $\gamma\in\Sign\cup\overline{\Sign}$ from Lemma \ref{l1} (if $r=2$ and $\gamma_1=\gamma_2$ then $\gamma\in\{(1,1),(2,2)\}$).

So assume now that $r\geq 3$. From Lemma \ref{l1} we have that $(\gamma_{r-1},\gamma_r)\in\Sign$. Also either $\gamma\in\{(1,1,1),(2,2,1)\}$ and then $\gamma\in\overline{\Sign}$ or $\gamma_i>\gamma_{i+1}$ for $1\leq i\leq r-2$. Assume now that for some $2\leq i\leq r-2$ we have $(\gamma_{i+1},\ldots,\gamma_r)\in\Sign$ and $(\gamma_i,\ldots,\gamma_r)\not\in\Sign$. Then $\gamma_i>\gamma_{i+1}$ and so, if $(\gamma_i,\ldots,\gamma_r)\not=(5,4,3,2,1)$, we can find $\beta$ with $\chi^\beta_{(\gamma_i,\ldots,\gamma_r)}\not\in\{0,\pm 1\}$ and $h_{2,1}^\beta=\gamma_i$. From $2\leq i\leq r-2$ it follows $r\geq 4$, so that $\gamma\not\in\{(1,1,1),(2,2,1)\}$ and then $\gamma_j>\gamma_i$ for $j<i$. Let
\[\delta:=(\beta_1+\gamma+1+\ldots+\gamma_{i-1},\beta_2,\beta_3,\ldots).\]
As $h_{2,1}^\delta=h_{2,1}^\beta=\gamma_i$ and $i\geq 2$ we have
\[\delta_1'\leq h_{2,1}^\delta+1=\gamma_i+1<1+\gamma_1\leq\delta_1,\]
so that $\delta$ is not self conjugate. From
\[\chi^\delta_\gamma=\chi^\beta_{(\gamma_i,\ldots,\gamma_r)}\not\in\{0,\pm 1\},\]
we then have a contradiction to $\gamma$ being an $A_n$-sign partition.

Assume now that $(\gamma_i,\ldots,\gamma_r)=(5,4,3,2,1)$. If $\gamma_{i-1}\geq 7$ let
\[\delta:=(4+\gamma_1+\ldots+\gamma_{i-1},4,4,3).\]
In this case $\delta_1'=4<11\leq \delta_1$, so that also in this case $\delta$ is not self conjugate. As $\gamma_{i-1}\geq 7$,
\[\chi^\delta_\gamma=\chi^{(4,4,4,3)}_{(5,4,3,2,1)}=-2\]
and then also in this case we have a contradiction. If instead $\gamma_{i-1}=6$ let
\[\delta:=(15+\gamma_1+\ldots+\gamma_{i-2},2,1,1,1,1).\]
Here too $\delta$ is not self conjugate and
\[\chi^\delta_\gamma=\chi^{(15,2,1,1,1,1)}_{(6,5,4,3,2,1)}=2,\]
again leading to a contradiction.

By induction $(\gamma_2,\ldots,\gamma_r)\in\Sign$. Assume now that $\gamma\not\in\Sign\cup\overline{\Sign}$. Then from Lemma \ref{t2} (as $\gamma\not\in\{(1,1,1),(2,2,1)\}$, so that $\gamma_1>\gamma_2$ from Lemma \ref{l1} in this case) there exists $\beta$ not self conjugate with $\chi^\beta_\gamma\not\in\{0,\pm 1\}$, again leading to a contradiction.

In particular if $\gamma$ is an $A_n$-sign partition then $\gamma\in\Sign\cup\overline{\Sign}$.

\section{Partitions consisting of odd distinct parts are not $A_n$-sign partitions}\label{s6}

Let $\gamma$ consists of odd distinct parts and let $\lambda$ be the self conjugate partition with diagonal hook lengths equal to the parts of $\gamma$. Then
\[\chi^{\lambda,\pm}_{\gamma,+}=\frac{\epsilon \pm\sqrt{\epsilon \gamma_1\ldots\gamma_r}}{2}\]
with $\epsilon\in\{\pm 1\}$.

If $\epsilon=-1$ clearly $\chi^{\lambda,\pm}_{\gamma,+}\not\in\{0,\pm 1\}$.

Assume now that $\epsilon=1$. From $\gamma$ being a partition of $n\geq 2$ it follows that $\gamma_1\cdots\gamma_r\geq 2$. In particular $\chi^{\lambda,+}_{\gamma,+}>1$.

Similarly $\chi^{\lambda,\pm}_{\gamma,-}$ are not both in $\{0,\pm 1\}$ and then $\gamma$ is not an $A_n$-sign partition. Together with Section \ref{s5} this prove that if $\gamma$ is an $A_n$-sign partition then
\[\gamma\in(\Sign\cup\overline{\Sign})\cap\{\mbox{even partitions not consisting of odd distinct parts}\}.\]

\section{Elements of $\Sign\cup\overline{\Sign}$ are $A_n$-sign partitions}\label{s7}

We will now prove that if
\[\gamma\in(\Sign\cup\overline{\Sign})\cap\{\mbox{even partitions not consisting of odd distinct parts}\}\]
then $\gamma$ is an $A_n$-sign partition. If $\gamma\in\Sign$ this is easily proved in the next theorem.

\begin{theor}
Let $\gamma\in\Sign$ be an even partition not consisting of odd distinct parts. Then $\gamma$ is an $A_n$-sign partition.
\end{theor}

\begin{proof}
From Theorem 1.3 of \cite{m2} we have that $\gamma$ is a sign partition. So $\chi^\beta_\gamma\in\{0,\pm 1\}$ for every $\beta\vdash n$. In particular $\chi_\gamma\in\{0,\pm1/2,\pm 1\}$ and then also $\chi_\gamma\in\{0,\pm 1\}$ for every irreducible character $\chi$ of $A_n$, that is $\gamma$ is an $A_n$-sign partition.
\end{proof}

For $\gamma\in\overline{\Sign}$ the proof is more complicated. It can be checked that $(1,1,1)$, $(2,2)$, $(2,2,1)$, $(5,4,3,2,1)$, $(6,5,4,1)$, $(6,3,2,1,1)$, $(9,4,3,2,1)$ and $(10,5,3,2,1)$ are $A_n$-sign partitions by looking at the corresponding characters tables. For the other elements of $\overline{\Sign}$ we will use the following lemmas.

\begin{lemma}\label{l6}
If $\gamma\in\overline{\Sign}$ does not consist of odd distinct parts, then $\gamma$ is an $A_n$-sign partition if and only if $\chi^\beta_\gamma\in\{0,\pm 1\}$ for every $\beta\vdash n$ not self conjugate with at least two $\gamma_1$-hooks.
\end{lemma}

\begin{proof}
Let $\gamma\in\overline{\Sign}$. Then $\gamma$ is an even partition. From $\gamma\in\overline{\Sign}$ we have that $(\gamma_2,\ldots,\gamma_r)\in\Sign$ and that $|\gamma|\leq 3\gamma_1$. In particular any $\beta\vdash n$ has at most 3 $\gamma_1$-hooks. As $(\gamma_2,\ldots,\gamma_r)\in\Sign$ and then it is a sign partition from Theorem 1.3 of \cite{m2}, it follows that $\chi^\beta_\gamma\in\{0,\pm1,\pm2,\pm3\}$ for every $\beta\vdash n$.

As $\gamma$ does not consist of odd distinct parts, it is then enough to prove that $\chi^\beta_\gamma\in\{0,\pm 1\}$ for every $\beta\vdash n$ not self conjugate with at least two $\gamma_1$-hooks, as then $\chi_\gamma\in\{0,\pm1/2,\pm1,\pm3/2\}$, and so also $\chi_\gamma\in\{0,\pm1\}$, for every irreducible character $\chi$ of $A_n$.
\end{proof}

The next lemma is a generalization of Lemma 4.1 of \cite{m2}.

\begin{lemma}\label{l4}
Let $\gamma=(\gamma_1,\ldots,\gamma_r)$ be a partition. Assume that $(\gamma_2,\ldots,\gamma_r)$ is a sign partition an that $\gamma_2+\ldots+\gamma_r<2a$. If $\beta$ is a partition of $n$ for which $\chi^\beta_\gamma\not\in\{0,\pm 1\}$ then $\beta$ has two $\gamma_1$-hooks. Also if $\delta$ is obtained from $\beta$ by removing a $\gamma_1$-hook then $\chi^\delta_{(\gamma_2,\ldots,\gamma_r)}\not=0$. In particular each such $\delta$ has a $\gamma_2$-hook.
\end{lemma}

\begin{proof}
By assumption
\[n=|\gamma|=\gamma_1+\gamma_2+\ldots+\gamma_r<3a.\]
It follows that any partition of $n$ has at most two $\gamma_1$-hooks. As
\[\chi^\beta_\gamma=\sum_{(i,j):h_{i,j}^\beta=\gamma_1}\pm\chi^{\beta\setminus R_{i,j}^\beta}_{(\gamma_2,\ldots,\gamma_r)}\]
and $(\gamma_2,\ldots,\gamma_r)$ is a sign partition by Theorem 1.3 of \cite{m2}, so that $\chi^{\beta\setminus R_{i,j}^\beta}_{(\gamma_2,\ldots,\gamma_r)}\in\{0,\pm 1\}$ for each $(i,j)\in[\beta]$ with $h_{i,j}^\beta=\gamma_1$, the lemma follows.
\end{proof}

We will now prove that all remaining elements of $\overline{\Sign}$ are $A_n$-sign partitions if they do not consists of odd distinct parts.

\begin{theor}
If $a\geq 7$ then $(a,a-1,4,1)$ is an $A_n$-sign partition.
\end{theor}

\begin{proof}
As $|(a,a-1,4,1)|=2a+4<3a$, so that any partition of $(a,a-1,4,1)$ has at most two $a$-hooks, we only need to consider, from Lemma \ref{l6}, partitions of $2a+4$ with 2 $a$-hooks. Let $\beta$ be any such partition. Then the $a$-core of $\beta$ is a partition of 4 and so $\beta_{(a)}\in\{(4),(3,1),(2,2),(2,1,1),(1,1,1,1)\}$. As $\chi^{\lambda'}_\rho=\pm\chi^\lambda_\rho$ and $(\lambda')_{(q)}=(\lambda_{(q)})'$ for $\lambda,\rho$ partitions and $q\geq 1$, we can assume that $\beta_{(a)}\in\{(4),(3,1),(2,2)\}$.

Assume first that $\beta_{(a)}=(4)$. From any $a$-core we can obtain exactly $a$ different partitions by adding an $a$-hook to it. In this case they are given by
\[\{(a+4)\}\cup\{(4, i,1^{a-i}):1\leq i\leq 4\}\cup\{(a-i,5,1^{i-1}):1\leq i\leq a-5\}.\]
If $\mu$ and $\nu$ are the partitions obtained from $\beta$ by removing an $a$-hook, we can assume from Lemma \ref{l4} that
\[\mu,\nu\in\{(a+4),(4,1^a),(5,5,1^{a-6}),(a-1,5)\}\]
as the other partition do not have $(a-1)$-hooks. From $a\geq 7$ and then
\[\chi^{(4,1^a)}_{(a-1,4,1)},\chi^{(a-1,5)}_{(a-1,4,1)}=\pm\chi^{(4,1)}_{(4,1)}=0\]
we can further assume that $\mu,\nu\in\{(a+4),(5,5,1^{a-6})\}$. Since $\mu\not=\nu$ and as we can recover $\beta$ from the $a$-cores and $a$-quotients of $\mu$ and $\nu$ (there exists a unique such partition $\beta$), we obtain that $\beta=(a+4,6,1^{a-6})$ and then that
\[\chi^\beta_{(a,a-1,4,1)}=-\chi^{(5,5,1^{a-6})}_{(a-1,4,1)}+(-1)^{a-6}\chi^{(a+4)}_{(a-1,4,1)}=-(-1)^{a-6}\chi^{(5)}_{(4,1)}+(-1)^{a-6}=0.\]

Assume next that $\beta_{(a)}=(3,1)$. In this case
\begin{eqnarray*}
\mu,\nu&\in&\{(a+3,1),(a,4),(3,3,2,1^{a-4}),(3,2,2,1^{a-3}),(3,1^{a+1})\}\\
&&\hspace{12pt}\cup\{(a-i,4,2,1^{i-2}):2\leq i\leq a-4\}.
\end{eqnarray*}
Again from Lemma \ref{l4} we can assume that $\mu$ and $\nu$ have an $(a-1)$-hook and so
\[\mu,\nu\in\{(a+3,1),(a,4),(3,2,2,1^{a-3}),(3,1^{a+1}),(4,4,2,1^{a-6}),(a-2,4,2)\}.\]
Since
\begin{eqnarray*}
\chi^{(a+3,1)}_{(a-1,4,1)},\chi^{(3,1^{a+1})}_{(a-1,4,1)},\chi^{(4,4,2,1^{a-6})}_{(a-1,4,1)}&=&\pm\chi^{(4,1)}_{(4,1)}=0,\\
\chi^{(a-2,4,2)}_{(a-1,4,1)}&=&\chi^{(3,1,1)}_{(4,1)}=0
\end{eqnarray*}
we can assume that $\mu,\nu\in\{(a,4),(3,2,2,1^{a-3})\}$, that is $\beta=(a,4,3,1^{a-3})$. In this case
\[\chi^\beta_{(a,a-1,4,1)}\!=(-1)^{a-3}\chi^{(a,4)}_{(a-1,4,1)}+\chi^{(3,2,2,1^{a-3})}_{(a-1,4,1)}\!=(-1)^{a-2}\chi^{(3,2)}_{(4,1)}+(-1)^{a-3}\chi^{(3,2)}_{(4,1)}=0.\]

Let now $\beta_{(a)}=(2,2)$. Then
\begin{eqnarray*}
\mu,\nu&\in&\{(a+2,2),(a+1,3),(2,2,2,1^{a-2}),(2,2,1^a)\}\\
&&\hspace{12pt}\cup\{(a-i,3,3,1^{i-2}):2\leq i\leq a-3\}.
\end{eqnarray*}
From $\mu$ and $\nu$ having an $(a-1)$-hook it follows
\[\mu,\nu\in\{(a+2,2),(2,2,1^a),(a-2,3,3),(3,3,3,1^{a-5})\}.\]
As $(a+2,2)'=(2,2,1^a)$ and $(a-2,3,3)'=(3,3,3,1^{a-5})$ and from $\beta$ being self-adjoint if $\mu=\nu'$ (as the partitions obtained from $\beta'$ by removing an $a$-hook are $\mu'=\nu$ and $\nu'=\mu$ in this case) we can assume that $\mu\in\{(a+2,2),(2,2,1^a)\}$ and $\nu\in\{(a-2,3,3),(3,3,3,1^{a-5})\}$. From $\chi^\lambda_\rho=\pm\chi^{\lambda'}_\rho$, we can further assume that $\mu=(a+2,2)$. Then $\beta\in\{(a+2,a-1,3),(a+2,4,3,1^{a-5})\}$. From
\begin{eqnarray*}
\chi^{(a+2,a-1,3)}_{(a,a-1,4,1)}&\!=\!&-\chi^{(a-2,3,3)}_{(a-1,4,1)}-\chi^{(a+2,2)}_{(a-1,4,1)}=-\chi^{(2,2,1)}_{(4,1)}-\chi^{(3,2)}_{(4,1)}=0,\\
\chi^{(a+2,4,3,1^{a-5})}_{(a,a-1,4,1)}&\!=\!&-\chi^{(3^3,1^{a-5})}_{(a-1,4,1)}+(-1)^{a-4}\chi^{(a+2,2)}_{(a-1,4,1)}=(-1+1)(-1)^{a-4}\chi^{(3,2)}_{(4,1)}=0
\end{eqnarray*}
it then follows that $(a,a-1,4,1)$ is an $A_n$-sign partition.
\end{proof}

\begin{theor}
If $a\geq 8$ then $(a,a-3,2,1,1)$ is an $A_n$-sign partition.
\end{theor}

\begin{proof}
Since $|(a,a-3,2,1,1)|=2a+1$ we only need to consider not self-adjoint $\beta\vdash 2a+1$ with two $a$-hooks from Lemma \ref{l6}. In this case $\beta_{(a)}=(1)$. Let $\mu,\nu$ be obtained from $\beta$ by removing these $a$-hooks. Then
\[\mu,\nu\in\{(a+1),(1^{a+1})\}\cup\{(a-i,2,1^{i-1}):1\leq i\leq a-2\}.\]
From Lemma \ref{l4} we can assume that $\mu$ and $\nu$ have an $(a-3)$-hook, so that
\[\mu,\nu\in\{(a+1),(1^{a+1}),(a-1,2),(a-3,2,1,1),(4,2,1^{a-5}),(2^2,1^{a-3})\}.\]
As $a-3>2$
\[\chi^{(a-1,2)}_{(a-3,2,1,1)},\chi^{(2^2,1^{a-3})}_{(2,1,1)}=\pm\chi^{(2,2)}_{(2,1,1)}=0\]
we can further assume that
\[\mu,\nu\in\{(a+1),(1^{a+1}),(a-3,2,1,1),(4,2,1^{a-5})\}.\]
From $(a+1)'=(1^{a+1})$ and $(a-3,2,1^2)'=(4,2,1^{a-5})$ we can, as in the previous theorem, assume that $\mu=(a+1)$ and $\nu\in\{(a-3,2,1^2),(4,2,1^{a-5})\}$, so that $\beta\in\{(a+1,a-2,1,1),(a+1,5,1^{a-5})\}$. From $a>5$ it follows
\begin{eqnarray*}
\chi^{(a+1,a-2,1,1)}_{(a,a-3,2,1,1)}&\!\!\!=\!\!\!&\chi^{(a+1)}_{(a-3,2,1,1)}-\chi^{(a-3,2,1,1)}_{(a-3,2,1,1)}=1+\chi^{(1^4)}_{(2,1,1)}=0,\\
\chi^{(a+1,5,1^{a-5})}_{(a,a-3,2,1,1)}&\!\!\!=\!\!\!&(-1)^{a-5}\chi^{(a+1)}_{(a-3,2,1,1)}-\chi^{(4,2,1^{a-5})}_{(a-3,2,1,1)}=(1-1)(-1)^{a-5}\chi^{(4)}_{(2,1,1)}=0
\end{eqnarray*}
it then follows from Lemma \ref{l6} that $(a,a-3,2,1,1)$ is an $A_n$-sign partition.
\end{proof}

\begin{theor}
If $a\geq 12$ then $(a,a-5,3,2,1)$ is an $A_n$-sign partition.
\end{theor}

\begin{proof}
As $|(a,a-5,3,2,1)|=2a+1$ we only need to consider not self-adjoint $\beta\vdash 2a+1$ with two $a$-hooks from Lemma \ref{l6}. Also in this case $\beta_{(a)}=1$, so that if $\mu$ and $\nu$ are obtained from $\beta$ by removing an $a$-hook and they have an $(a-5)$-hook (as we can apply Lemma \ref{l4}), we have that
\begin{eqnarray*}
\mu,\nu&\in&\{(a+1),(1^{a+1}),(a-1,2),(a-2,2,1),(a-3,2,1^2),(a-5,2,1^4),\\
&&\hspace{12pt}(6,2,1^{a-7}),(4,2,1^{a-5}),(3,2,1^{a-4}),(2^2,1^{a-3})\}.
\end{eqnarray*}
From $a-5>4$ and then
\begin{eqnarray*}
\chi^{(a-1,2)}_{(a-5,3,2,1)},\chi^{(4,2,1^{a-5})}_{(a-5,3,2,1)}&=&\pm\chi^{(4,2)}_{(3,2,1)}=0,\\
\chi^{(a-2,2,1)}_{(a-5,3,2,1)},\chi^{(3,2,1^{a-4})}_{(a-5,3,2,1)}&=&\pm\chi^{(3,2,1)}_{(3,2,1)}=0,\\
\chi^{(a-3,2,1^2)}_{(a-5,2,1^2)},\chi^{(2^2,1^{a-3})}_{(a-5,3,2,1)}&=&\pm\chi^{(2,2,1,1)}_{(3,2,1)}=0
\end{eqnarray*}
we can assume that
\[\mu,\nu\in\{(a+1),(1^{a+1}),(a-5,2,1^4),(6,2,1^{a-7})\}\]
and then again that $\mu=(a+1)$ and $\nu\in\{(a-5,2,1^4),(6,2,1^{a-7})\}$, so $\beta\in\{(a+1,a-4,1^4),(a+1,7,1^{a-7})\}$. From $a>11$ we have that
\begin{eqnarray*}
\chi^{(a+1,a-4,1^4)}_{(a,a-5,3,2,1)}&\!\!\!=\!\!\!&\chi^{(a+1)}_{(a-5,3,2,1)}-\chi^{(a-5,2,1^4)}_{(a-5,3,2,1)}=1+\chi^{(1^6)}_{(3,2,1)}=0,\\
\chi^{(a+1,7,1^{a-7})}_{(a,a-5,3,2,1)}&\!\!\!=\!\!\!&(-1)^{a-7}\chi^{(a+1)}_{(a-5,3,2,1)}-\chi^{(6,2,1^{a-7})}_{(a-5,3,2,1)}=(1-1)(-1)^{a-7}\chi^{(6)}_{(3,2,1)}=0.
\end{eqnarray*}
In particular $(a,a-5,3,2,1)$ is an $A_n$-sign partition.
\end{proof}

\begin{theor}
If $b+1\leq a\leq 2b-2$ then $(a,b,a-b+1)$ is an $A_n$-sign partition if it does not consists of odd distinct parts.
\end{theor}

\begin{proof}
From $b+1\leq a\leq 2b-2$ it follows that $a>b>a-b+1>1$. Again from Lemma \ref{l6} we only need to consider partitions $\beta$ of $|(a,b,a-b+1)|=2a+1$ with two $a$-hooks which are not self-adjoint. From Lemma \ref{l4} in this case $\mu$ and $\nu$ have a $b$-hook, so that
\begin{eqnarray*}
\mu,\nu&\in&\{(a+1),(1^{a+1})\}\cup\{(a-i,2,1^{i-1}):1\leq i\leq a-b-2,\,\,i=a-b,\\
&&\hspace{12pt}i=b-1\mbox{ or }b+1\leq i\leq a-2\}.
\end{eqnarray*}
From $b>a-b+1$ it follows that any partition of $a+1$ has at most one $b$-hook, so that
\[\chi^{(a-i,2,1^{i-1})}_{(b,a-b+1)}=\chi^{(a-b-i,2,1^{i-1})}_{(a-b+1)}=0\]
for $1\leq i\leq a-b-2$ and similarly for $b+1\leq i\leq a-2$. As in the previous theorems we can assume that $\mu=(a+1)$ and $\nu\in\{(b,2,1^{a-b-1}),(a-b+1,2,1^{b-2})\}$, that is $\beta\in\{(a+1,b+1,1^{a-b-1}),(a+1,a-b+2,1^{b-2})\}$. As 
\begin{eqnarray*}
\chi^{(a+1,b+1,1^{a-b-1})}_{(a,b,a-b+1)}&\!\!\!\!=\!\!\!\!&(-1)^{a-b-1}\chi^{(a+1)}_{(b,a-b+1)}\!-\!\chi^{(b,2,1^{a-b-1})}_{(b,a-b+1)}\!=\!(-1)^{a-b-1}\!+\!\chi^{(1^{a-b+1})}_{(a-b+1)}\!=\!0,\\
\chi^{(a+1,a-b+2,1^{b-2})}_{(a,b,a-b+1)}&\!\!\!\!=\!\!\!\!&(-1)^{b-2}\chi^{(a+1)}_{(b,a-b+1)}\!-\!\chi^{(a-b+1,2,1^{b-2})}_{(b,a-b+1)}\!=\!(-1)^{b-2}(1\!-\!\chi^{(a-b+1)}_{(a-b+1)})\!=\!0
\end{eqnarray*}
we have that $(a,b,a-b+1)$ is an $A_n$-sign partition if it does not consists of odd distinct parts and $b+1\leq a\leq 2b-2$.
\end{proof}

\begin{theor}
If $b+2\leq a\leq 2b$ then $(a,b,a-b-1,1)$ is an $A_n$-sign partition if it does not consists of odd distinct parts.
\end{theor}

\begin{proof}
Here $a>b>a-b-1\geq 1$, as $b+2\leq a\leq 2b$. From Lemma \ref{l6} we only need to compute $\chi^\beta_{(a,b,a-b-1,1)}$ for $\beta\vdash 2a$ not self-adjoint with two $a$-hook. If $\mu$ and $\nu$ are obtained from such a $\beta$ by removing an $a$-hook then $\mu$ and $\nu$ are $a$-hooks. From $a\leq 2b$ it follows that $\mu$ and $\nu$ have at most one $b$-hook. If such a $b$-hook exists then
\[\mu,\nu\in\{(a-i,1^i):0\leq i\leq a-b-1\mbox{ or }b\leq i\leq a-1\}.\]
For $1\leq i\leq a-b-2$
\[\chi^{(a-i,1^i)}_{(b,a-b-1,1)}=\chi^{(a-i-b,1^i)}_{(a-b-1,1)}=0\]
and similarly for $b+1\leq i\leq a-2$. So from Lemma \ref{l4} we can assume that
\[\mu,\nu\in\{(a),(b+1,1^{a-b-1}),(a-b,1^b),(1^a)\}.\]
Also here we can assume $\mu=(a)$ and $\nu\in\{(b+1,1^{a-b-1}),(a-b,1^b)\}$, that is $\beta\in\{(a,b+2,1^{a-b-2}),(a,a-b+1,1^{b-1})\}$. From
\begin{eqnarray*}
\chi^{(a,b+2,1^{a-b-2})}_{(a,b,a-b-1,1)}&\!\!\!\!=\!\!\!\!&(-1)^{a-b-2}\chi^{(a)}_{(b,a-b-1,1)}\!-\!\chi^{(b+1,1^{a-b-1})}_{(b,a-b-1,1)}\!=\!(-1)^{a-b}\!-\!\chi^{(1^{a-b})}_{(a-b-1,1)}\!=\!0,\\
\chi^{(a,a-b+1,1^{b-1})}_{(a,b,a-b-1,1)}&\!\!\!\!=\!\!\!\!&(-1)^{b-1}\chi^{(a)}_{(b,a-b-1,1)}\!-\!\chi^{(a-b,1^b)}_{(b,a-b-1,1)}\!=\!(-1)^{b-1}(1\!-\!\chi^{(a-b)}_{(a-b-1,1)})\!=\!0
\end{eqnarray*}
it follows that $(a,b,a-b-1,1)$ is an $A_n$-sign partition if it does not consists of odd distinct parts.
\end{proof}

\section*{Acknowledgements}

The author thanks Christine Bessenrodt and Michael Cuntz for questions and discussion about sign conjugacy classes of $A_n$, which lead to writing this paper.

\end{document}